\documentclass{amsart}
\usepackage{amssymb}  
\usepackage{amsmath} 
\usepackage{amsthm}  
\usepackage{bm} 
\usepackage{stmaryrd} 

\advance\hoffset by -1in
\advance\voffset by -1in

\newtheorem{thm}{Theorem}
\newtheorem{lem}[thm]{Lemma}

\newtheorem{prob}[thm]{Problem}

\title[Embedding central extensions into wreath products]{
Embedding central extensions of simple \\ linear groups into wreath products}

\author{Andrei V. Zavarnitsine}
\address{\textup{\scriptsize
Andrei V. Zavarnitsine\\
Sobolev Institute of Mathematics\\
4, Koptyug av.\\
630090, Novosibirsk, Russia
}}
\email{zav@math.nsc.ru}

\begin{document}
\begin{abstract} We find a criterion for the embedding of a nonsplit central extension of $\mathrm{PSL}_n(q)$ with
kernel of prime order into
the permutation wreath product that corresponds to the action on the projective space.


\end{abstract}
\maketitle

\section{Introduction}

A group $B$ included in a short exact sequence of groups
$$
1\to A\stackrel{\iota}{\to}B\to C\to 1
$$
is called an extension of $A$ by $C$ and denoted by $A.C$ in general, or by $A\leftthreetimes C$ if it is split. The extension
$B$ is {\em central} if $\iota(A)\leqslant \mathrm{Z}(B)$.

We write $\mathbb{Z}_n$ for a cyclic group of order $n$ and $(m,n)$ for $\mathrm{gcd}(m,n)$.

This note concerns the following

\begin{prob}\label{mp} Let $G=\operatorname{PSL}_n(q)$, where $q$ is a prime power, and let $r$ be a prime divisor of $(n,q-1)$.
Does the  permutation wreath product $\mathbb{Z}_r\wr_\rho G$ contain a subgroup isomorphic to the nonsplit central extension
$\mathbb{Z}_r.G$, where $\rho$ is the natural permutation representation of $G$ on the points of the projective space
$\mathbb{P}^{n-1}(q)$? \end{prob}

We remark that the nonsplit extension $\mathbb{Z}_r.G$ mentioned in Problem \ref{mp} is unique up to isomorphism and is a quotient of
$\operatorname{SL}_n(q)$. This problem is a generalization of the one raised in \cite[p.\,67]{mas}, where the case $n=r$ is
considered. The case $n=2$ was studied in \cite{zse}, where it was shown that the
embedding holds if and only if $q\equiv -1\mod 4$. We generalize this result by proving the following

\begin{thm}\label{main}
In the notation of Problem \ref{mp}, the nonsplit central extension $\mathbb{Z}_r.G$ is embedded into $\mathbb{Z}_r\wr_\rho G$ if
and only if $r$ does not divide $(q-1)/(n,q-1)$.
\end{thm}

The main method that we use is based on some cohomological considerations and is similar to that of \cite{zse}.

\section{Preliminaries}

Let $G$ be a group and let $L,M$ be right $G$-modules. Suppose
\begin{equation}\label{emb}
0\to L \to M
\end{equation}
and
$$
1\to M \to E \stackrel{\pi}\to G \to 1
$$
are exact sequences of modules and groups, where the conjugation action of $E$ on $M$ agrees with the $G$-module structure, i.\,e.
$m^e=m\cdot\pi(e)$ for all $m\in M$ and $e\in E$, and we identify $M$ with its image in $E$. A subgroup $S\leqslant E$ such that
\begin{equation}\label{hpro}
S\cap M = L, \qquad SM = E,
\end{equation}
where we also identify $L$ with its image in $M$, which is itself an extension of $L$ by $G$, will be called a {\em subextension}
of $E$ that corresponds to the embedding (\ref{emb}).

It is known \cite{gru} that the equivalence classes of extensions of $L$ by $G$ are in a one-to-one correspondence with (thus are
{\em defined by}) the elements of the second cohomology group $H^2(G,L)$. Furthermore, the sequence (\ref{emb}) gives rise to a
homomorphism
\begin{equation}\label{hhom}
H^2(G,L)\stackrel{\varphi}{\to} H^2(G,M).
\end{equation}

\begin{lem}\cite[Lemma 2]{zse} \label{bas} Let $L,M$ be $G$-modules and $E$ an extension as specified above. Let
$\bar\gamma\in H^2(G,M)$ be the element that defines $E$. Then
the set of elements of $H^2(G,L)$ that define the subextensions $S$ of $E$ corresponding to the embedding {\rm (\ref{emb})}
coincides with $\varphi^{-1}(\bar\gamma)$, where $\varphi$ is the induced homomorphism {\rm (\ref{hhom})}. In particular, $E$ has
such a subextension $S$ if and only if $\bar\gamma\in \operatorname{Im}\varphi$.
\end{lem}

We now present a slight generalization of the argument in \cite[Section 7]{zse}.

Denote by $\mathbb{F}_q$ a finite field of order $q$. The order of a group element $g$ will be denoted by $|g|$.

Let $G$ be a finite group, $X$ a set, and let $\rho$ be a permutation representation of $G$ on $X$. For a prime $p$, we consider
the permutation $\mathbb{F}_pG$-module $V$ that corresponds to $\rho$ with basis (identified with) $X$ and its trivial submodule
$I$ spanned by $\sum_{x\in X}x$. Clearly, the wreath product $\mathbb{Z}_p\wr_\rho G$ is the natural split extension $V\leftthreetimes G$.

\begin{lem}\label{bur} In the above notation, if a central extension $S=\mathbb{Z}_p.G$ is a subextension of $V\leftthreetimes G$ that
corresponds to the embedding of $\mathbb{F}_pG$-modules $I\to V$ then $S$ has no element $s$ that satisfies the following three
conditions
\begin{enumerate}
  \item[$(i)$] $|s|=p^2$,
  \item[$(ii)$] $|g|=p$, where $g\in G$ is the image of $s$ under the natural epimorphism $S \to G$.
  \item[$(iii)$] $\rho(g)$ has a fixed point on $X$.
\end{enumerate}
\end{lem}
\begin{proof}
  Assume to the contrary that $s$ is such an element. Denote $t=\sum_{x\in X}x \in I$. Since $S=I.G$, we have $s^p=ct$ for a nonzero $c\in \mathbb{F}_p$,
  and since $S$ is a subextension of $V\leftthreetimes G$ with respect to $I\to V$, there exists $v\in V$ such that $s=gv$.
  Therefore, we have $ct=s^p=(gv)^p=vh$, where
  $$
  h=1+g+\ldots+g^{p-1}.
  $$
  Let $x\in X$ be a fixed point of $\rho(g)$. We can write $v=a_xx+w$, for some $a_x\in \mathbb{F}_p$, where $w=\sum_{y\in X\setminus\{x\}}a_yy$. Clearly, $wh$ is a linear combination
  of elements of $X\setminus\{x\}$, and
  $$
  (a_xx)h=a_x(x+\ldots +x)=a_x\, px=0.
  $$
  Hence, the coefficient of $x$ in $ct=vh$ is zero, which contradicts $c\ne 0$.
\end{proof}

\section{A permutation module for $\operatorname{PSL}_n(q)$}

We henceforth denote $G=\operatorname{PSL}_n(q)$ and fix a prime divisor $r$ of $(n,q-1)$.
The natural permutation action $\rho$ of $G$
on the points of the projective space $\mathcal{P}=\mathbb{P}^{n-1}(q)$ gives rise to a permutation $\mathbb{F}_rG$-module $V$. As every permutation
module, $V$ has a trivial submodule $I$ spanned by $\sum_{x\in \mathcal{P}}x$, and the augmentation submodule $V_0$ that consists of the elements
$\sum_{x\in \mathcal{P}}a_x x$ with $\sum_{x}a_x =0$. Since $\operatorname{dim}V=1+q+\ldots+q^{n-1}\equiv 0 \pmod r$, we
have $I\leqslant V_0$, and the quotient $U=V_0/I$ is known \cite{mor} to be absolutely irreducible whenever $n \geqslant 3$. It was noticed by
various authors \cite{bra,bur} that $U$ is one of the few examples of modules with $2$-dimensional $1$-cohomology, namely we have:

\begin{lem} \label{1coh} In the above notation, $H^1(G,U)\cong\mathbb{Z}_r^2$, whenever $n\ge 3$.
\end{lem}

We will also require the $1$-cohomology of $V$.

\begin{lem} \label{v1coh} Let $V$ be the above-defined permutation module. Then we have
$$
H^1(G,V)\cong\left\{
\begin{array}{rl}
  \mathbb{Z}_r, & r \ \text{divides}\ (q-1)/(n,q-1), \\
  0, & \text{otherwise}.
\end{array}
\right.
$$
\end{lem}
\begin{proof} We will assume that $n\geqslant 3$, as the claim holds for $n=2$ by \cite[Lemma 12]{zse}.
Since the action of $G$ on $\mathcal{P}$ is transitive, we have $V\cong T^G$, where $T$ is the principal $\mathbb{F}_rH$-module for
a point stabilizer $H$. By Shapiro's lemma \cite[\S 6.3]{wei}, we have $H^1(G,V)\cong H^1(H,T)\cong \operatorname{Hom}(H/H',T)$.
The structure of $H$ is known \cite[Section 2]{maz} and has the shape $\operatorname{ASL}_{n-1}(q).\mathbb{Z}_{(q-1)/(n,q-1)}$.
Since $n\ge 3$ and $(n,q-1)>1$, the group $\operatorname{ASL}_{n-1}(q)$ is perfect and so $H/H'\cong \mathbb{Z}_{(q-1)/(n,q-1)}$.
As $T$ is cyclic of order $r$, the claim follows.
\end{proof}

\section{Proof}

We can now prove Theorem \ref{main}. Due to \cite{zse}, we may assume that $n\geqslant 3$. We denote by $S$ the nonsplit central extension
$\mathbb{Z}_r. G$. Since $G$ is simple, the only possibility for $S$ to be a subgroup of the extension $V\leftthreetimes G$ is if $S$ is
its subextension, and since $I$ is the unique trivial submodule of $V$, this subextension must be with respect to the embedding $I\to V$.
Being split, the extension $V\leftthreetimes G$ is defined by the zero element of $H^2(G,V)$. Hence, Lemma \ref{bas} implies that
all subextensions of $V\leftthreetimes G$ with respect to $I\to V$
are defined by the elements of $\operatorname{Ker}\varphi$, where $\varphi$ is the induced homomorphism
\begin{equation}\label{phi}
  H^2(G,I)\stackrel{\varphi}{\to} H^2(G,V).
\end{equation}
The short exact sequence of modules
$$
0\to I \to V \to V^0 \to 0,
$$
where $V^0\cong V/I$, gives rise to the long exact sequence
\begin{equation}\label{lc}
H^1(G,I)\to H^1(G,V)\stackrel{\alpha}{\to} H^1(G,V^0)\stackrel{\delta}{\to} H^2(G,I)\stackrel{\varphi}{\to} H^2(G,V),
\end{equation}
which implies that $\operatorname{Ker}\varphi=\operatorname{Im}\delta$. Observe that $H^1(G,I)\cong \operatorname{Hom}(G/G',I)=0$,
since $G$ is simple. Therefore, the map $\alpha$ in (\ref{lc}) is an embedding, and $\operatorname{Ker}\varphi\cong H^1(G,V^0)/H^1(G,V)$.

The structure of $V$, see \cite[Lemma 2]{mor}, allows us to include $V^0$  in the nonsplit short exact sequence
$$
0\to U \to V^0 \to I \to 0,
$$
which gives rise to the exact sequence
\begin{equation}\label{lz}
H^0(G,V^0)\to H^0(G,I)\to H^1(G,U)\to H^1(G,V^0)\to H^1(G,I).
\end{equation}
Now, $H^0(G,V^0)=0$, since $V^0$ has no trivial submodules, and $H^1(G,I)=0$ as above.
Therefore, $H^1(G,V^0)\cong H^1(G,U)/H^0(G,I)$. Since $H^0(G,I)\cong \mathbb{Z}_r$, Lemma \ref{1coh} implies $H^1(G,V^0)\cong\mathbb{Z}_r$,
and so $\operatorname{Ker}\varphi$ is $0$ or $\mathbb{Z}_r$ according as $r$ divides $(q-1)/(n,q-1)$ or otherwise, by Lemma \ref{v1coh}.
It follows that the nonzero element of $H^2(G,I)$ that defines the nonsplit extension $S$ lies in $\operatorname{Ker}\varphi$ if and only if
$r$ does not divide $(q-1)/(n,q-1)$. By Lemma \ref{bas}, this completes the proof of the theorem.

In the case when $r$ divides $(q-1)/(n,q-1)$, we can also prove the nonembedding of $S$ into $V\leftthreetimes G$ in a different way. Suppose this is the case.
Then $\mathbb{F}_q$ has an element $a$ of  multiplicative order $r(n,q-1)$. Let $s$ be the image in $S$ of $\operatorname{diag}(a,a,\ldots,a,a^{1-n})$ under
the epimorphism $\operatorname{SL}_n(q)\to S$. We have $|s|=r^2$ and $|g|=r$, where $g$ is the image of $s$ under the epimorphism $S\to G$.
Observe that $\rho(g)$ has a fixed point on $\mathcal{P}$, because every diagonal element of $\operatorname{SL}_n(q)$ fixes a point on $\mathcal{P}$,
e.\,g. the projective image of the basis vector $(1,0,\ldots,0)$. Therefore,  $S$ cannot be a subextension of $V\leftthreetimes G$ by Lemma \ref{bur}.

\end{document}